\documentclass{article}
\usepackage[a4paper, left=0.8in, right=0.8in, top=1in, bottom=1in]{geometry}

\usepackage{graphicx}
\usepackage{caption}
\usepackage{enumerate,comment,mathtools,amsmath,amssymb,algorithm,amsthm}
\usepackage{algpseudocode}
\newtheorem{lemma}{Lemma}[section]
\newtheorem{theorem}[lemma]{Theorem}

\newtheorem{definition}{Definition}[section]
\newtheorem{proposition}{Proposition}[section]

\newtheorem{remark}{Remark}[section]
\newtheorem{example}{Example}[section]

\usepackage{xcolor}
\newcommand{\first}[1]{{\color{black}{#1}}}
\newcommand{\second}[1]{{\color{black}{#1}}}
\newcommand{\both}[1]{{\color{black}{#1}}}

\begin{document}

\title{The State-Dependent Riccati Equation in Nonlinear Optimal Control: Analysis and Numerical Approximation}

\author{
    Luca Saluzzi \thanks{Department of Mathematics and Computer Science, University of Ferrara, Italy (\texttt{luca.saluzzi@unife.it})}
}

\maketitle
\begin{abstract}
The State-Dependent Riccati Equation (SDRE) approach is extensively utilized in nonlinear optimal control as a reliable framework for designing robust feedback control strategies. This work provides an analysis of the SDRE approach, examining its theoretical foundations, error bounds, and numerical approximation techniques. We explore the relationship between SDRE and the Hamilton-Jacobi-Bellman (HJB) equation, deriving residual-based error estimates to quantify its suboptimality. Additionally, we introduce an optimal semilinear decomposition strategy to minimize the residual. From a computational perspective, we analyze two numerical methods for solving the SDRE: the offline–online approach and the Newton–Kleinman iterative method. Their performance is assessed through a numerical experiment involving the control of a nonlinear reaction-diffusion PDE. Results highlight the trade-offs between computational efficiency and accuracy, indicating better performance of the Newton–Kleinman approach in achieving stable and cost-effective solutions in the reported experiments.
\end{abstract}

\bigskip

\section{Introduction}

Optimal control of nonlinear dynamical systems is a fundamental problem in engineering, economics, and applied mathematics. The Hamilton-Jacobi-Bellman equation provides a rigorous framework for computing optimal feedback control laws, whose foundational work dates back to Bellman’s seminal contributions \cite{bellman1957}, establishing dynamic programming as a cornerstone of optimal control theory. However, solving the HJB equation in practical scenarios remains intractable due to its nonlinear and high-dimensional nature. To mitigate this issue, various approximation techniques have been developed, including sparse grids strategies \cite{GK16}, polynomial approximations \cite{kunisch2023learning,kalise2018polynomial}, applications of artificial neural networks \cite{Darbon_Langlois_Meng_2020,Kunisch_Walter_2021,Zhou_2021} and regression-type methods in tensor formats \cite{oster22,dolgov2023data}.
In this context, the State-Dependent Riccati Equation method has emerged as an effective approach, offering a computationally feasible alternative by leveraging local approximations of the value function through a Riccati-based formulation. The SDRE method extends the classical Linear Quadratic Regulator (LQR) framework to nonlinear systems by representing the system dynamics in a state-dependent linearized form. This transformation enables the application of Riccati-based feedback synthesis, preserving the stabilizing properties of LQR while maintaining adaptability to nonlinear dynamics. A key advantage of SDRE is its ability to provide suboptimal yet computationally efficient solutions, making it particularly useful in applications \cite{voos2006nonlinear,vaddi2009numerical,heiland2023low}.
In \cite{banks2007}, the authors developed a rigorous formulation of SDRE-based control, demonstrating its effectiveness in stabilizing nonlinear systems under mild regularity assumptions. Moreover, it has been shown that the SDRE feedback induces a stationary Hamiltonian with respect to the control, ensuring that as the state converges to zero, the necessary optimality conditions are asymptotically satisfied at a quadratic rate. For a comprehensive discussion on this topic, we refer to \cite{CIMEN,cloutier1996nonlinear}.
Despite its advantages, the SDRE approach is not without limitations. The choice of the semilinear representation can significantly impact the accuracy of the solution, and existing methods often rely on heuristics for selecting an appropriate decomposition. 
 Furthermore, numerical solvers for the SDRE require efficient methodologies for iteratively solving the associated Riccati equations. An initial attempt to mitigate the computational complexity is presented in \cite{beeler2000feedback}, where the authors introduce a power series expansion for the SDRE solution. This approach involves computing certain terms offline by solving a Riccati equation alongside a sequence of Lyapunov equations. However, for high-dimensional problems, this method may become impractical due to the computational burden and storage requirements associated with high-dimensional matrix equations. To address this challenge, the authors of \cite{alla2023state} proposed an \emph{offline-online} approach based on a first-order approximation, wherein the computational effort is partitioned into an intensive offline phase and a more efficient online stage. Another promising avenue involves solving the SDRE using the Newton-Kleinman method \cite{kleinman1968,saluzzi2025dynamical}, which capitalizes on the iterative nature of the approximation process.
 
 
\subsection{Contributions and Outline}
This paper aims to provide a deeper understanding of the SDRE method by analyzing its theoretical properties, deriving error bounds for the approximation, and exploring computational techniques for its implementation. Our contributions include:
\begin{itemize}
\item a derivation of error bounds for the SDRE approximation, quantifying the deviation from the optimal HJB solution;
\item an investigation of optimal semilinear representations that minimize the SDRE residual error,
\item a comparative study of numerical methods for solving the sequence of Riccati equations, highlighting the advantages of iterative methods.
\end{itemize}
The paper is structured as follows: Section 2 formulates the infinite-horizon optimal control problem and introduces key concepts related to the SDRE framework. Section 3 examines error bounds for the SDRE approach, focusing on the residual with respect to the HJB equation and the existence of an optimal semilinear formulation. Section 4 presents numerical techniques for solving the SDRE, including a comparative analysis of two computational methods applied to the control of a nonlinear reaction-diffusion PDE. Finally, Section 5 concludes with a discussion on potential directions for future research.

\section{The Infinite Horizon Optimal Control Problem}\label{sec:iocp}
In this section, we establish the formulation of the deterministic infinite-horizon optimal control problem, for which we aim to derive a feedback control law. We begin by presenting the optimal feedback synthesis via the Hamilton-Jacobi-Bellman framework.  

We consider a control-affine dynamical system governed by  
\begin{equation}\label{eq}
\left\{ \begin{array}{l}
\dot{y}(s) = f(y(s)) + B(y(s)) u(s), \quad s \in (0, +\infty),\\
y(0) = x \in \mathbb{R}^d.
\end{array} \right.
\end{equation}  
Here, $y: [0, +\infty) \to \mathbb{R}^d$ represents the system state, while $u: [0, +\infty) \to \mathbb{R}^m$ denotes the control input. The set of admissible control functions is given by \second{$\mathcal{U} = L^\infty ([0, +\infty); \mathbb{R}^m)$}. The system dynamics are described by the continuously differentiable functions $f: \mathbb{R}^d \to \mathbb{R}^d$ and $B: \mathbb{R}^d \to \mathbb{R}^{d \times m}$, satisfying $f({0})  = {0}$. When emphasizing the dependence of the control input on the initial state $x \in \mathbb{R}^d$, we explicitly denote it as $u(t;x)$.  

The objective is to minimize the following infinite-horizon cost functional:  
\begin{equation}\label{cost}
 J(u(\cdot \, ;x)) := \int_0^{+\infty} y(s)^{\top} Q y(s) + u^{\top}(s) R u(s)\,ds\,,
\end{equation}  
where $Q \in \mathbb{R}^{d \times d}$ is a symmetric positive semidefinite matrix, and $R \in \mathbb{R}^{m \times m}$ is a symmetric positive definite matrix. The goal is to determine an optimal control policy in feedback form, meaning that the control law depends solely on the current state of the system.  

To this end, we define the value function corresponding to a given initial condition $x \in \mathbb{R}^d$ as  
\begin{equation}
V(x) := \inf\limits_{u\in\mathcal{U}} J(u(\cdot;x))\,,
\label{VF}
\end{equation}  
which, according to standard dynamic programming principles, satisfies the HJB PDE for every $x \in \mathbb{R}^d$:  
\begin{equation}\label{HJB}
\min\limits_{u\in \second{\mathbb{R}^m} }\left\{(f(x)+B(x)u)^{\top} \nabla V(x) + x^{\top} Q x+ u^{\top} R u \right\}=0.
\end{equation}  

The HJB equation \eqref{HJB} is a first-order, fully nonlinear PDE defined over $\mathbb{R}^d$, making it computationally challenging, particularly when $d$ is large.
\second{We remark that the finiteness of the value function is not guaranteed a priori in nonlinear infinite-horizon problems. However, as shown in the subsequent section, under suitable assumptions on the data, the SDRE feedback yields a locally exponentially stable closed-loop system, from which it follows that the associated infinite-horizon cost is finite for all initial conditions in the corresponding region of attraction.
}

The optimal control minimizing the left-hand side of \eqref{HJB} can be derived explicitly as  
\begin{equation}\label{optc}
u^*(x) = -\frac{1}{2} R^{-1} B(x)^{\top} \nabla V(x)\,,
\end{equation}  
leading to the following unconstrained form of the HJB equation:  
\begin{equation}\label{hjbu}
\nabla V(x)^{\top} f(x) - \frac{1}{4} \nabla V(x)^{\top} B(x) R^{-1} B(x)^{\top} \nabla V(x) + x^{\top} Q x = 0.
\end{equation}  

The solution of the partial differential equation \eqref{hjbu} presents significant computational challenges due to the well-known \emph{curse of dimensionality}. Specifically, on a structured grid with $N$ grid points per dimension, the computational complexity scales as $O(N^d)$, rendering the problem intractable for high-dimensional systems. Rather than directly computing the optimal value function and the associated optimal trajectory, one may relax these requirements by seeking a \emph{suboptimal} yet stabilizing feedback control strategy. This consideration has contributed to the widespread adoption of the SDRE approach. In the following section, we introduce its formulation and examine its relationship with the HJB equation.

\subsection{State-Dependent Riccati Equation}
\label{sec:SDRE}
\both{Under the regularity assumptions introduced in the previous section, Proposition~1 in~\cite{cimen2010systematic} guarantees the existence of a state-dependent matrix function $A(x): \mathbb{R}^{d} \rightarrow \mathbb{R}^{d \times d}$ such that $f(x)=A(x)x$. Accordingly, the dynamical system \eqref{eq} can be written in the semilinear form}


\begin{equation}
	\dot{y}(t)  = A(y(t)) y(t) +B(y(t)) u(t).
	\label{semilinear}
\end{equation}
\both{
In dimension greater than one, this representation is non-unique, a feature that will be exploited in the following discussion. Detailed results concerning uniqueness can be found in \cite{cloutier1996nonlinear}.
}
In the special case where the matrix functions are constant, i.e., $A(y(t)) = A \in \mathbb{R}^{d\times d}$ and  $B(y(t)) = B \in \mathbb{R}^{d\times m}$, the problem reduces to the well-known Linear Quadratic Regulator.
\second{

To extend the classical LQR framework to the state-dependent setting, we introduce stabilizability and detectability for non-constant matrix functions defined on a set $\Omega \subseteq \mathbb{R}^d$.

\begin{definition}
    Let $\Omega \subseteq \mathbb{R}^d$ be a nonempty set.  
    The pair $(A(\cdot), B(\cdot))$ is said to be stabilizable on $\Omega$ if for every $x \in \Omega$ there exists a feedback matrix $K(x) \in \mathbb{R}^{m \times d}$  
    such that the matrix $A(x) - B(x) K(x)$ is stable, that is, all its eigenvalues lie strictly in the open left half of the complex plane.
\end{definition}

\begin{definition}
    Let $\Omega \subseteq \mathbb{R}^d$ be a nonempty set.  
    The pair $(A(\cdot), C)$ is said to be detectable on $\Omega$ if the pair $(A^\top(\cdot), C^\top)$ is stabilizable on $\Omega$ 
    in the sense of the previous definition.
\end{definition}

When $A(\cdot)$ and $B(\cdot)$ are constant, these definitions reduce to the classical LQR notions of stabilizability and detectability.} In this setting, when the pair $(A,B)$ is stabilizable and the pair $(A,Q^{1/2})$ is detectable, the optimal feedback control for the LQR is given by:
\begin{equation*}
	u(y) = -R^{-1} B^{\top} P y,
\end{equation*}
where $P\in\mathbb{R}^{d\times d}$ is the unique positive definite solution of the Continuous Algebraic Riccati Equation (CARE)
\begin{align*}
	A^\top P + P A
	-P BR^{-1}B^\top P +Q = 0\ .
\end{align*}
The SDRE method generalizes this approach by allowing the system matrices to depend on the state, leading to the control law
\begin{equation}
	u_S(y) = -R^{-1} B^{\top}(y) P(y)y\,,
	\label{control_sdre}
\end{equation}
where $P(y)$ now satisfies a State-Dependent Riccati Equation:
\begin{align}
	A^\top(y) P(y) +  P(y) A(y)-P(y)B(y)R^{-1}B^\top(y)P(y)+Q & = 0,
	\label{sdre}
\end{align}
with the matrices $A(y)$ and $B(y)$ evaluated at the state $y$. A practical implementation of SDRE-based control for nonlinear stabilization employs a receding horizon strategy. In Algorithm \ref{alg:sdre} the method is sketched. Specifically, at each discrete state \(y_S(t_k)\), the matrices in equation \eqref{sdre} are fixed, and the corresponding Riccati equation is solved for 
\( P(y_S(t_k))\). The control input  \(u_S(y_S(t_k))\) from equation \eqref{control_sdre} is then applied over a short horizon to evolve the system dynamics. This process is repeated at the next state \(y_S(t_{k+1})\) iteratively. 

\begin{algorithm}[htbp]
\caption{The reduced horizon-SDRE method}
\label{alg:sdre}{
\begin{algorithmic}[1]
\Statex{\textbf{Input:} $A(\cdot), B(\cdot), Q,R$, initial conditions $ y_0$, time discretization $\{t_i\}_{i=0}^{N_T}$ }
\Statex{\textbf{Output:} Controlled trajectory $\{y_S(t_i)\}_{i=0}^{N_T}$}

\State{$y_S(t_0):= y_0$}
\For{$n = 0, \dots, N_T-1$}

\State{Compute $P(y_S(t_{n}))$ from \eqref{sdre}}
\State{Build the control $u_S(y_S(t_{n}))$ from \eqref{control_sdre}}
\State{Integrate the system \eqref{semilinear} to obtain $y_S(t_{n+1})$}       
\EndFor 
\end{algorithmic}}
\end{algorithm}

Under suitable stability assumptions, it can be shown that the closed-loop system governed by the feedback law \eqref{control_sdre} exhibits local asymptotic stability, as formalized in the following theorem from \cite{banks2007}.

\begin{theorem}
    Assume that the system \eqref{eq} is such that \( f(x) \) and \( \frac{\partial f(x)}{\partial x_j} \) (for \( j = 1, \dots, n \)) are continuous in \( x \) for all \( \|x\| \leq \hat{r} \), where \( \hat{r} > 0 \).
Assume further that \( A(x) \) and \( B(x) \) are continuous and that the pair $(A(x),B(x))$ is stabilizable and the pair $(A(x),Q^{1/2})$ is detectable in some nonempty neighborhood of the origin \( \mathcal{O} \subseteq B_{\hat{r}}(0) \). Then, the system with the control given by \eqref{control_sdre} is locally asymptotically stable.
\label{thm:stable}
\end{theorem}

\second{Indeed, the authors obtain a stronger result. The system controlled by the feedback law \eqref{control_sdre} is locally exponentially stable, \emph{i.e.}, there exist constants $r>0$, $C>0$, and $\lambda>0$ such that
$\|y_S(t)\| \le C \|y_S(0)\| e^{-\lambda t}$, for all $\|y_S(0)\|\le r,\; t \in [0,+\infty)$.
This exponential decay implies that the quadratic running cost is integrable over $[0,+\infty)$, and therefore the value function is finite for all initial states within the region of attraction of the stabilizing feedback, as anticipated in Section~\ref{sec:iocp}.} 

Beyond stabilization, an alternative perspective considers the deviation of the SDRE approach from the optimality conditions dictated by HJB equation. To analyze this discrepancy, we introduce the SDRE-based approximation of the value function:
\begin{equation}
    V_{S}(x) = x^\top P(x) x,
    \label{v_sdre}
\end{equation}  
where \( P(x) \) satisfies the SDRE \eqref{sdre}. The feedback control \eqref{control_sdre} is derived from the optimal control formula \eqref{optc}, under the assumption that the term involving the gradient of \( P(x) \) is negligible.
However, accounting for this term leads to the modified feedback law:

\begin{equation}
\tilde{u}_{S}(x) = -\frac{1}{2} R^{-1} B(x)^T \nabla V_S (x),
\label{utilde_sdre}
\end{equation}
where $\nabla V_S (x) = 2P(x) x + \varphi(x)$,
with the correction term $\varphi$ defined as
$$
[\varphi(x)]_i =  x^\top P_{x_i}(x) x, \quad i =1, \dots, d,
$$
with \( P_{x_i}(x) = \frac{\partial P(x)}{\partial x_i} \). By employing similar analytical techniques as those utilized in the proof of Theorem \ref{thm:stable} in \cite{banks2007}, one can establish that the system dynamics under the control $\tilde{u}_{S}(x)$ exhibit local asymptotic stability, as formally stated in the following result.

\begin{theorem}

Under the assumptions of Theorem \ref{thm:stable}, the system with the control given by \eqref{utilde_sdre} is locally \second{exponentially} stable.
\label{thm:stable2}
\end{theorem}

\begin{proof}
    Under the feedback control
$\tilde{u}_{S}(x)$, the closed-loop dynamics becomes
$$
    \dot{y} = A(y)y - B(y) R^{-1} B(y)^T P(y)y - \frac{1}{2} B(y) R^{-1} B(y)^T \varphi(y).
    \label{closed_loop}
$$
Equivalently, this can be expressed as
$$
 \dot{y} = (A_0-B_0 R^{-1}B_0^\top P_0)y + h(y),
$$
where $A_0=A(0)$, $B_0 = B(0)$, $P_0 = P(0)$ and 
\second{
\begin{equation*}
\begin{aligned}
h(y)
&= \bigl(A(y) - A_0\bigr)\,y
   - \Bigl(B(y) R^{-1} B(y)^\top P(y) - B_0 R^{-1} B_0^\top P_0\Bigr)\,y \\
&\quad - \frac{1}{2}\, B(y) R^{-1} B(y)^\top \varphi(y).
\end{aligned}
\end{equation*}
}

It follows directly that
$$
    \lim_{y \to 0} \frac{\|h(y)\|}{\|y\|} = 0,
$$
since all the state-dependent matrices in $h(y)$ are continuous and the additional term involving $\varphi$ constitutes a perturbation of higher order in $y$. As a result, the stability analysis proceeds within the same framework as employed in the proof of Theorem \ref{thm:stable} from \cite{banks2007}.
 
\end{proof}

Now, given the SDRE ansatz \eqref{v_sdre}, it is possible to compute its residual with respect to the HJB equation \eqref{hjbu} just inserting $\nabla V_S(x)$ into \eqref{hjbu}, obtaining 
$$
    x^\top \big( P(x) A(x) + A^\top(x) P(x) - P(x) B(x) R^{-1} B^\top(x) P(x) + Q \big) x +E(x) = 0,
    $$
    with
    \begin{equation}
        E(x) = \varphi(x)^\top   \left([A(x)- B(x)R^{-1}B(x)^\top P(x)] x - \frac{1}{4} B(x) R^{-1} B(x)^\top \varphi(x) \right).
        \label{residual}
    \end{equation}
The first term in the expression corresponds to the SDRE \eqref{sdre}. As a result, it is either exactly zero in the theoretical formulation or negligibly small when considering numerical approximations, depending on the accuracy of the numerical method employed. The function $E(x)$ quantifies the effect of the sub-optimality associated with the SDRE approach. While the feedback control $u_S$ provides a stabilizing suboptimal control, the function $\tilde{u}_{S}$ accounts for the correction term $E(x)$, making it a refined alternative when optimality is a concern. Indeed, minimizing $E(x)$  aligns the feedback law $\tilde{u}_{S}$ with the optimal control solution derived from the HJB equation. 
We note that fixing the semilinear form $A(x)$, the residual $E(x)$ can be explicitly determined by computing the derivatives \( \{P_{x_i}(x)\}_i \). Whenever we want to stress the dependence on the semilinear form, we write $P(x;A(x))$ and $E(x;A(x))$.

To obtain the derivatives \( \{P_{x_i}(x)\}_i \), we differentiate the SDRE \eqref{sdre} with respect to \(x_i\), yielding the following Lyapunov equation:

\begin{equation}
P_{x_i}(x)A_{cl}(x) + A_{cl}^\top(x)P_{x_i}(x) 
+\Lambda_i(x) = 0,\label{Lyap}
\end{equation}
where 
\begin{align*}
\Lambda_i(x) =\, &P(x)A_{x_i}(x) + A^\top_{x_i}(x)P(x) \nonumber\\[1mm]
& - P(x) (B_{x_i}(x)R^{-1}B^\top(x) + B(x)R^{-1}B^\top_{x_i}(x))P(x)
\end{align*}
and
\[
A_{cl}(x) = A(x) - B(x)R^{-1}B^\top(x)P(x),
\]
which is referred to as the \emph{closed-loop} system matrix.

Equation \eqref{Lyap} defines a Lyapunov equation for \(P_{x_i}(x)\). Solving this equation for each $i \in \{1, \dots, d\}$ allows for the computation of the residual term \eqref{residual}, thereby quantifying the deviation introduced by the omitted gradient terms.
In particular, if the residual vanishes for all $x$, then the SDRE formulation is equivalent to the HJB equation. Conversely, if the residual does not vanish, one may incorporate $E(x)$ into the running cost to show that the function $V_S(x)$ satisfies a modified HJB equation. Consequently, $V_S(x)$ adheres to a dynamic programming principle, as formalized in the following result.

\begin{proposition}
Define the augmented running cost by
\begin{equation}\label{eq:augmented_cost}
\tilde{\ell}(x,u) = x^\top Q x + u^\top R u + E(x),
\end{equation}
where $E(x)$ is given by \eqref{residual}.
Then, the function $V_S(x)$
satisfies:
\begin{itemize}
    \item The modified HJB equation:
    \begin{equation}
    \min_{u\in \mathbb{R}^m} \left\{ \nabla V_S(x)^\top\Bigl(A(x)x+B(x)u\Bigr) + \tilde{\ell}(x,u) \right\} = 0.
    \label{hjb_sdre}
        \end{equation}
    \item The dynamic programming principle (DPP):
       \begin{equation}
    V_S(x) = \inf_{u \in \mathcal{U}} \left\{ \int_0^T \tilde{\ell}\bigl(x(s),u(s)\bigr)\,ds + V_S\bigl(x(T)\bigr) \right\},
        \label{dpp_sdre}
        \end{equation}
    for every \(T\ge 0\).
\end{itemize}
\end{proposition}

\begin{proof}
As $E(x)$ is defined as the residual of $V_S(x)$ with respect to the HJB equation \eqref{HJB}, it satisfies a modified HJB equation in which the running cost is given by the augmented formulation \eqref{eq:augmented_cost}. The DPP formulation in \eqref{dpp_sdre} follows directly from the classical theory of HJB equations; see, for instance, \cite{BCD97} for further details.
 
\end{proof}

\begin{example}[Linear Quadratic Regulator]
A fundamental and straightforward example is the LQR. In this case, the system matrices $A(x)=A$ and $B(x)=B$ are constant, implying that the solution to the SDRE is also constant. Specifically, under these conditions, the HJB equation and the SDRE approach (which, in this case, reduces to the standard LQR formulation) coincide, yielding the exact optimal value function and the corresponding optimal feedback control law.
    
\end{example}

\begin{example}[Van Der Pol oscillator]
\label{ex:vdp}
We consider the dynamics

\[
\dot{x}=A(x)x+B(x)u,
\]
with
\[
A(x)=\begin{pmatrix}0 & 1\\[1mm]-1 & -0.5\,(1-x_1^2)\end{pmatrix},\quad B(x)=\begin{pmatrix}0\\ x_1\end{pmatrix}.
\]
and cost functional \eqref{cost} with
\[
Q=\begin{pmatrix}0 & 0\\[1mm]0 & 1\end{pmatrix},\quad R=1.
\]

A direct computation shows that the constant choice
\[
P(x)=I=\begin{pmatrix}1&0\\[1mm]0&1\end{pmatrix}
\]
satisfies the SDRE \eqref{sdre}. 

Since \(P(x)\) is constant, its derivatives vanish and the residual (which involves the derivatives $\{P_{x_i}(x)\}_i$) is identically zero. In this scenario, despite the state-dependent nature of the system matrices, full optimality is achieved as a consequence of the SDRE solution remaining constant.

\end{example}

\section{Error Bounds based on the Residual}

We have established that the residual of the SDRE solution with respect to the HJB equation can be quantified using equation \eqref{residual}, which relies on solving the Lyapunov equations for the derivatives $\{P_{x_i}\}_i$. In this section, our objective is to derive an error bound between the SDRE approximation \eqref{v_sdre} and the optimal value function \eqref{VF}, based on the residual $E(x)$. For a discounted optimal control problem, where the cost functional is defined as
\begin{equation*}\label{cost2}
 J(u(\cdot \,;x)) := \int_0^{+\infty} e^{-\rho t}\left(y(s)^{\top} Q y(s) + u^{\top}(s) R u(s)\right) \,ds\,,
\end{equation*} 

with $\rho>0$, it follows that the dynamic programming operator exhibits $L^\infty$-contractive (see, $e.g.$, \cite{dolcetta1983discrete}). This property enables the derivation of an error estimate of the form
$$
\Vert \tilde{V} - V \Vert_{\infty} \le \frac{\Vert E_{\tilde{V}} \Vert_{\infty}}{\rho},
$$
where $\tilde{V}$ represents an approximation of the value function, and $E_{\tilde{V}}$ denotes its residual with respect to the HJB equation (see \cite{grune1997adaptive,albert2002posteriori} for related approaches). However, this methodology is not applicable in our setting, as we consider an undiscounted framework where $\rho = 0$. Instead, we employ the dynamic programming principle \eqref{dpp_sdre} for the perturbed value function $V_S(x)$ in conjunction with the local asymptotic stability of the SDRE-controlled trajectory. A similar approach has been explored in the extended abstract \cite{grune2018hamiltonian}, where the authors establish a qualitative error bound based on the residual analysis.

\begin{proposition}
Under the assumptions of Theorem \ref{thm:stable}, the following error bound holds in some nonempty neighborhood of the origin
    \begin{equation}
|V_S(x) - V(x)| \leq \max\Bigl\{\int_0^{\infty} |E(y^*(t))| dt, \int_0^{\infty} |E(\tilde{y}_{S}(t))| dt\Bigr\}, 
\label{eq:bound}
\end{equation}
where $y^*(t)$ denotes the state trajectory corresponding to the optimal control \eqref{optc}, and $\tilde{y}_{S}(t)$  is the trajectory associated with the control \eqref{utilde_sdre}.

Furthermore, assume that both controlled trajectories are exponentially stable and remain in a neighborhood $\Omega$ of the origin. Then there exists a constant $C>0$ such that
\begin{equation}
|V_S(x)-V(x)|
\le
\frac{C}{\lambda}\,\|E\|_{*,\Omega}\,\|x\|,
\qquad x\in\Omega,
\label{eq:error_weighted}
\end{equation}
where
\[
\|E\|_{*,\Omega}:=\sup_{z\in\Omega\setminus\{0\}} \frac{|E(z)|}{\|z\|}.
\]

\end{proposition}

\begin{proof}

Applying the DPP for the optimal value function, we obtain that for any \(T\ge 0\)
\begin{equation}\label{eq:Vopt}
V(x)=\int_0^T \ell\bigl(y^*(s),u^*(s)\bigr)ds+V\bigl(y^*(T)\bigr),
\end{equation}
where \(y^*(t)\) is the state trajectory corresponding to the optimal control \(u^*(t)\).

Now, applying the same control \(u^*(\cdot)\) into the DPP for the SDRE approximation, we obtain
\begin{equation}\label{eq:VSDRE}
V_S(x)\le \int_0^T \Bigl[\ell\bigl(y^*(s),u^*(s)\bigr)+E\bigl(y^*(s)\bigr)\Bigr]ds+V_S\bigl(y^*(T)\bigr).
\end{equation}
By subtracting equation \eqref{eq:Vopt} from equation \eqref{eq:VSDRE} and utilizing the local asymptotic stability of the trajectory, the terminal terms vanish in the limit as $T \rightarrow + \infty$, leading to
\[
V_S(x)-V(x)\le \int_0^\infty E\bigl(y^*(s)\bigr)ds.
\]

A similar bound for $V(x)-V_S(x)$ can be derived by considering the control $\tilde{u}_{S}$ given in \eqref{utilde_sdre}, obtaining the bound \eqref{eq:bound}.


Now, by definition of $\|E\|_{*,\Omega}$,
\[
|E(z)|\le \|E\|_{*,\Omega}\|z\|,
\qquad z\in\Omega.
\]
Hence, using exponential stability,
\[
|E(y^*(t;x))|
\le
\|E\|_{*,\Omega}\|y^*(t;x)\|
\le
C \|E\|_{*,\Omega} e^{-\lambda t}\|x\|,
\]
and similarly for $\tilde y_S(t;x)$. Therefore,
\[
\int_0^\infty |E(y^*(t;x))|\,dt
\le
C \|E\|_{*,\Omega}\|x\|
\int_0^\infty e^{-\lambda t}\,dt
=
\frac{C}{\lambda}\|E\|_{*,\Omega}\|x\|.
\]
The same estimate holds along $\tilde y_S(t;x)$, which proves
\eqref{eq:error_weighted}.

\end{proof}

The norm $\|E\|_{*,\Omega}$ in the error bound \eqref{eq:error_weighted} can be evaluated within a compact domain $\Omega$ provided that both the controlled trajectories $y^*(t)$ and $\tilde{y}_S(t)$ remain within $\Omega$.

The error bound given in \eqref{eq:bound} is purely theoretical, as it relies on knowledge of the optimal trajectory $y^*(t)$. However, if $V_S(x) \le V(x)$, the inequality is obtained considering the integral of the residual along the trajectory computed via $\tilde{u}_{S}$ and it is a practical indicator, since it is fully computable. This approach involves computing the controlled trajectory $\tilde{y}_{S}(t)$ and evaluating the residual based on the methodology outlined in Section \ref{sec:SDRE}.  In the following example, we show that, for the given scenario, this indicator, derived from the SDRE trajectory, guarantees the validity of the upper bound.

\begin{example}(Van Der Pol oscillator)

Let us consider again the setting of Example \ref{ex:vdp}. Let us consider a different choice for semilinear form
\first{
\[
A(x)=\begin{pmatrix} -x_2 & 1+x_1\\[1mm]-1 & -0.5\,(1-x_1^2)\end{pmatrix},\quad B(x)=\begin{pmatrix}0\\ x_1\end{pmatrix}.
\]
}
In this scenario, the semilinear form does not represent an optimal solution, and consequently, \(E(x)\) remains nonzero for certain values of $x\in \mathbb{R}^2$. Let us examine the initial condition \(x^0 = (-0.5, 0.5)\) and compute the trajectory governed by the controller $\tilde{u}_S$. Subsequently, we evaluate the error \(|V(x) - V_S(x)|\) along the SDRE-controlled trajectory, together with the error bound as specified in \eqref{eq:bound}. The results are presented in the left panel of Figure \ref{pics:vdp}. The integral of the residuals computed along the optimal trajectory \(y^*(t)\) and the trajectory \(\tilde{y}_{S}(t)\) exhibit similar behavior. Therefore, only the integral corresponding to the latter is presented in Figure \ref{pics:vdp}.
 It is observed that as the state approaches the origin, the error diminishes towards zero, while the error bound remains closely aligned with the peaks of the error, indicating its effectiveness as an error predictor. Finally, in the right panel of Figure \ref{pics:vdp}, the trajectory obtained through the SDRE controller is depicted, where it is evident that the system stabilizes towards the origin.

\begin{figure}[H]	
\begin{center}    
	\includegraphics[width=0.49\textwidth]{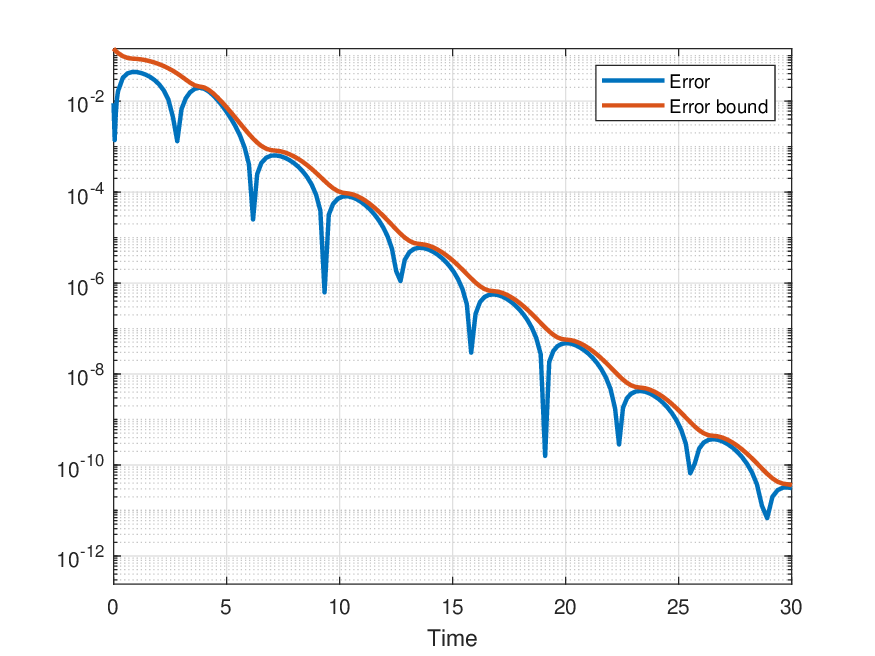}	
 	\includegraphics[width=0.49\textwidth]{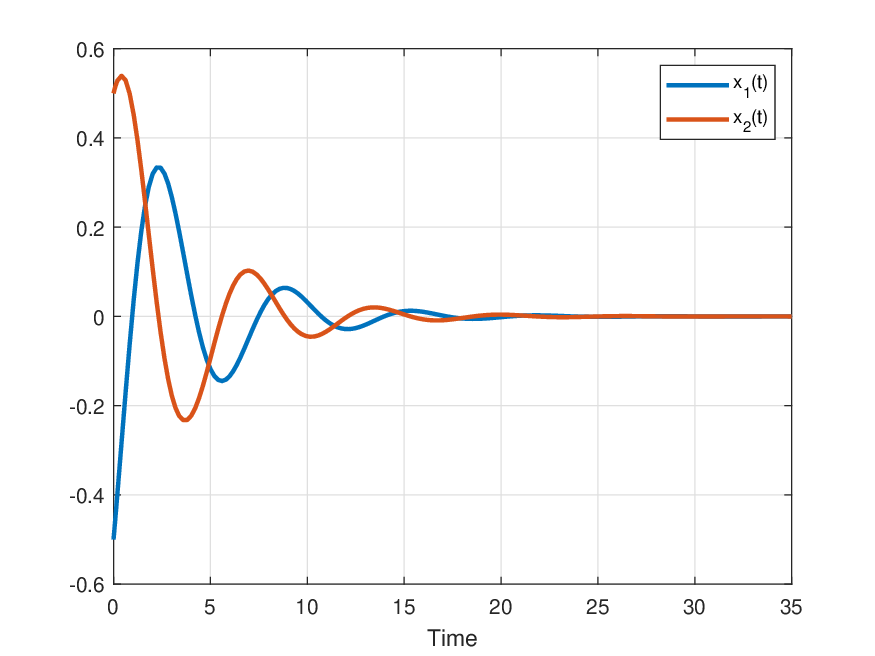}	
\caption{Van Der Pol example.  \emph{Left}: The temporal evolution of both the error between the value function and its SDRE approximation and the corresponding error bound as given in \eqref{eq:bound} along the trajectory induced by the SDRE controller. \emph{Right}: controlled solution using the SDRE controller.}
\label{pics:vdp}
\end{center}

\end{figure}

\end{example}

\subsection{Optimal Semilinear Form}

In the previous sections, we introduced the SDRE methodology, which relies on the semilinear representation \eqref{semilinear} of the dynamical system. This representation plays a fundamental role, as different choices of semilinear forms can lead to varying outcomes in the control strategy.
A natural question that arises is whether there exists at least one semilinear decomposition for which the residual term $E(x)$ vanishes entirely. This is a critical consideration, as the presence of $E(x)$ contributes to the sub-optimality of the SDRE approach. The following result provides an affirmative answer to this question by establishing the existence of such a semilinear form, under the assumption that the residual exhibits a sign change for two distinct semilinear formulations.

\begin{proposition}
\first{Let us consider the assumptions of Theorem \ref{thm:stable} and a fixed $x\in\mathbb{R}^n$.}
    Consider a fixed base semilinear form $A_0(x)$ satisfying $A_0(x)x=f(x)$, then the general semilinear form is parameterized as
\[
A(x)=A_0(x)+Z(x),
\]
where the free matrix $Z(x)$ is any matrix in
\[
\mathcal{Z}(x)=\{Z\in\mathbb{R}^{n\times n} : Zx=0\}.
\]
Given the mapping 
 \[
   \Phi:\mathcal{Z}(x) \to \mathbb{R},\quad \Phi(Z)=E(x;A_0+Z),
   \]
suppose that $\Phi(0_{n \times n}) \le 0$ and there exists $Z_1 \in\mathcal{Z}(x)$ such that $\Phi(Z_1)>0$.
   Then there exists a choice of $Z^*\in\mathcal{Z}(x)$ such that
\[
E(x;A_0(x)+Z^*)=0.
\]
   
\end{proposition}

\begin{proof}
    Since every semilinear form $A(x)$ satisfying $A(x)x=f(x)$ can be written as
   \[
   A(x)=A_0(x)+Z(x),
   \]
   with $Z(x)x=0$, the freedom in choosing $A(x)$ is equivalent to the freedom in choosing $Z(x)$ in the vector space
   \[
   \mathcal{Z}(x)=\{Z\in\mathbb{R}^{n\times n}: Zx=0\}.
   \]
   Note that if $x\neq 0$, then $\mathcal{Z}(x)$ is a linear subspace of dimension $n(n-1)$. With the above parameterization, the residual in the HJB equation can be viewed as a function
   \[
   \Phi:\mathcal{Z}(x)\to\mathbb{R},\quad \Phi(Z)=E(x;A_0(x)+Z).
   \]
  This mapping is continuously differentiable, as it results from the composition of smooth functions. Indeed, the residual $E(x;A_0(x)+Z)$ involves smooth operations and the SDRE solution $P(x;A_0(x)+Z)$ depends smoothly on its parameters (see, $e.g.$, \cite{ran1988parameter,sontag2013mathematical}).
  In the case $\Phi(0_{n \times n})=0$, the result follows by setting $Z^*=0_{n \times n}$. Otherwise, assume \(\Phi(0_{n \times n}) < 0\). By hypothesis, there exists \( Z_1 \in \mathcal{Z}(x) \) with \(\Phi(Z_1) > 0\).  
   Define the function \( f : [0,1] \to \mathbb{R} \) by
   $ 
   f(t) = \Phi(t Z_1).
   $
   Since scalar multiplication in $\mathcal{Z}(x)$ is continuous and \(\Phi\) is continuous, the composition \(f\) is continuous on \([0,1]\).  
   Notice that:
   \[
   f(0) = \Phi(0_{n \times n}) < 0 \quad \text{and} \quad f(1) = \Phi(Z_1) > 0.
   \]
   By the intermediate value theorem, there exists a \( t_0 \in (0,1) \) such that
   $
   f(t_0) = 0.
   $
   Setting \( Z^* = t_0 Z_1 \), which belongs to \(\mathcal{Z}(x)\) because \(\mathcal{Z}(x)\) is a linear subspace and is closed under scalar multiplication, we obtain
   \[
   \Phi(Z^*) = \Phi(t_0 Z_1) = 0,
   \]
   completing the proof.
    
   \end{proof}


The aforementioned result is not merely theoretical, but also provides a systematic approach for constructing the optimal semilinear form. Given a reference semilinear form $A_0(x)$ and an alternative semilinear form $A_0(x)+Z(x)$, with $Z(x) \in \mathcal{Z}(x)$ and the corresponding residual $E(x)$ exhibits a sign change, an optimal semilinear form can be obtained by considering the continuous path $A_0(x)+tZ(x)$ for $t \in [0,1]$ and identifying the point where the residual vanishes.
For instance, fixing two pairs of indices $(i_1,j_1)$ and $(i_1,j_2)$, one may consider the following class of matrices in the set $\mathcal{Z}(x)$
\begin{equation}
[Z(x)]_{i,j}= \begin{cases}  x_{j_2} & (i,j)=(i_1,j_1),  \\
- x_{j_1} & (i,j)=(i_1,j_2), \\
0 & otherwise .
\end{cases}
\label{eq:Zx}
\end{equation}
It is straightforward to check that $Z(x)x = 0$ for all $x \in \mathbb{R}^d$. 
\first{A theoretical characterization of when such semilinear forms exist is generally difficult. However, the construction becomes possible whenever the admissible class $\mathcal{Z}(x)$ contains at least one perturbation $Z(x)$ for which the associated residual, evaluated along the path $A_0(x) + t Z(x)$, changes sign. In practice, this can be assessed by introducing several (possibly random) perturbations of the baseline semilinear form and verifying whether such a sign change occurs.
}
If identifying the sign-changing semilinear form proves challenging, the problem can be reformulated from an optimization perspective. Specifically, for a fixed $x \in \mathbb{R}^d$ and a base semilinear form $A_0(x)$, one can seek to minimize the residual norm by solving the optimization problem:
\begin{equation}
\min_{Z \in \mathcal{Z}(x)} | E(x;A_0+Z)|^2.
\label{min_E}
\end{equation}
One approach to tackling this problem is through a parametrization of the vector space $\mathcal{Z}(x)$. In \cite{jones2020} the authors attempt to solve the minimization problem \eqref{min_E} by considering a first-order Taylor expansion of the parametrization for $\mathcal{Z}(x)$, while in \cite{dolgov2022optimizing}, the authors restrict the minimization to a class of $\mathcal{Z}(x)$ formed using the class of linear perturbation described by \eqref{eq:Zx}. In both cases, the proposed numerical experiments demonstrate that the technique can generate a more stabilizing feedback. However, these methodologies are affected by dimensionality constraints and are primarily suitable for low-dimensional systems. \second{
\begin{remark}
Pursuing optimal or near-optimal semilinear forms remains possible in high dimension by exploiting additional structure.  
One option is to restrict $Z(x)$ to sparse, block-diagonal, or low-rank, whose complexity scales linearly rather than quadratically with the state dimension; see, e.g., \cite{benner2015survey}.  
Another strategy relies on projection-based model reduction: one computes an optimal semilinear form in a reduced-order model and lifts it back to the full space, following standard reduced-order optimal control techniques \cite{hinze2005podcontrol}.   

Alternatively, randomized and sampling-based approaches can be employed: by probing a small number of random directions in $\mathcal{Z}(x)$, one can efficiently detect sign changes of the residual $E(x;A_0+tZ)$ even in very large dimensions, leveraging ideas from high-dimensional randomized numerical linear algebra \cite{halko2011finding}.  
Finally, data-driven surrogate models may approximate the mapping $Z \mapsto E(x;A_0+Z)$ at fixed $x$, reducing the high-dimensional optimization to a learned low-complexity problem. Such approaches have been successful in the approximation of high-dimensional HJB-type equations \cite{onken2022neural,dolgov2023data}.
\end{remark}
}

\begin{example}[Allen-Cahn equation]
   We consider the following nonlinear Allen-Cahn PDE with homogeneous Neumann boundary conditions:
\begin{equation*}
\left\{ \begin{array}{l}
\partial_t y(t,x) = \sigma \partial_{xx} y(t,x) + y(t,x) (1-y(t,x)^2) + u(t,x),  \\
 y(0,x)=y_0(x),
\end{array} \right.
\end{equation*}
where $x \in [0,1]$ and $t \in (0,+\infty)$. The associated cost functional is given by
$$
J(u,y_0) = \int_0^{\infty}  \int_0^1 (|y(t,x)|^2 +  \tilde{\gamma}\|u(t,x)|^2) dx \, dt \,.
$$
By discretizing the PDE using finite difference schemes with $d$ grid points, we obtain the following system of ODEs:
\begin{align*}
\dot{y}(s)=A_0(y) y(s) + u(s),
\end{align*}
where
\begin{align*}
A_0(y) = \sigma A^{\Delta} + I_d -  diag(y \odot y),\quad y \in \mathbb{R}^d, 
\end{align*}
with $\odot$ denoting the Hadamard \first{(element-wise)} product, $I_d \in \mathbb{R}^{d \times d}$ being the identity matrix and $A^{\Delta}$ arising from the discretization of the Laplace operator with homogeneous Neumann boundary conditions. We set $\sigma = 10^{-q}$,  $d=100$, $\tilde{\gamma} = 0.1$ and $y_0(x) = \cos(\pi x)$. Our objective is to determine an optimal semilinear form $A(x)$ such that the residual $E(y_0;A(x))$ vanishes. To this end, we introduce a linear perturbation as defined in \eqref{eq:Zx}, specifically choosing $i_1=j_1=1$ and $j_2=2$, meaning that we perturb the first two entries of the first row of the semilinear form. We consider the parameterized form $A(x;\alpha) = A_0(x)+ \alpha Z(x)$, with $\alpha \in \mathbb{R}$. Left panel of Figure \ref{pics:AC} illustrates the behavior of the residual $E(y_0;\alpha)$ as the parameter $\alpha$ varies over the interval $[-10,10]$. The results exhibit a quadratic-like behavior, with at least one zero in the interval $(9,10)$, where a change of sign is observed. The root can be computed using the MATLAB function \texttt{fzero}, yielding $\alpha^* \approx 9.23 $, at which the residual attains the value $5.5 \cdot 10^{-15}$.
A similar result is obtained when considering the initial condition $y_0(x) = \sin (\pi x)$, as shown in the right panel of Figure \ref{pics:AC}, where the zero occurs at $\alpha^* \approx 5.74 $.

\begin{figure}[H]	
\begin{center}    
	\includegraphics[width=0.49\textwidth]{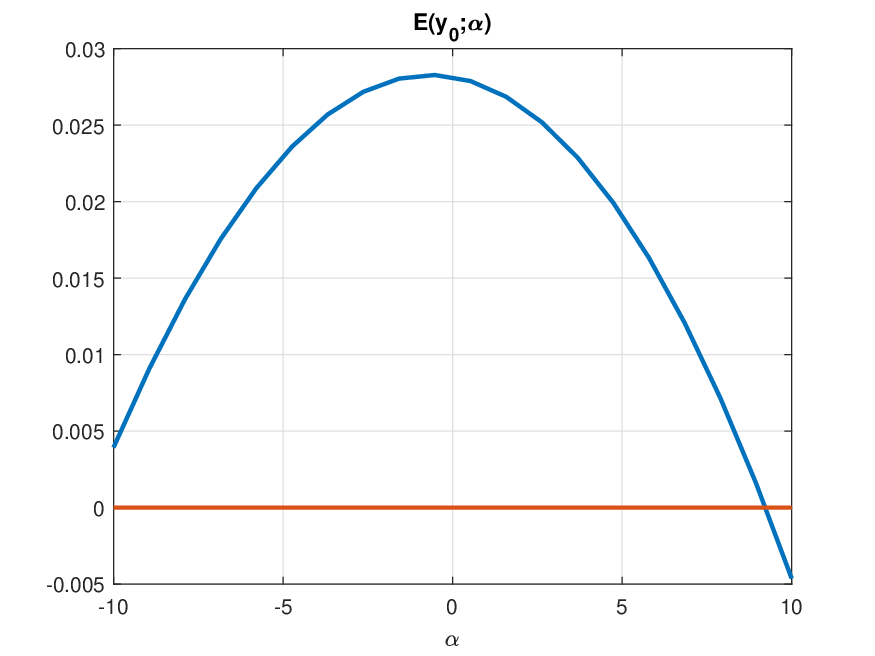}	
    	\includegraphics[width=0.49\textwidth]{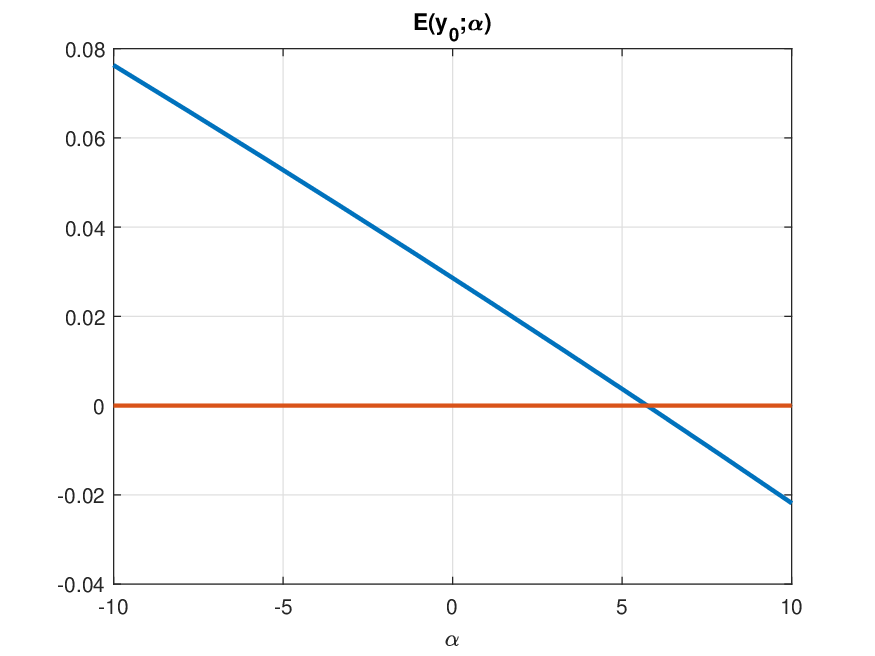}	
	
\caption{Allen-Cahn example. Variation of the residual as a function of the parameter $\alpha$ in the semilinear form, demonstrating the existence of at least one zero for $y_0(x) = \cos (\pi x)$ (left) and $y_0(x) = \sin (\pi x)$ (right).}
\label{pics:AC}
\end{center}

\end{figure}
\end{example}

\section{Numerical Methods for the Approximation of the SDRE}

In this section, we examine the numerical approximation of the SDRE methodology, highlighting that its implementation necessitates the repeated solution of the CAREs throughout the time integration process. This study investigates and compares two distinct numerical approaches for solving the sequence of CAREs, both of which adhere to Algorithm \ref{alg:sdre}.
For simplicity, we focus on the case $B(x)=B$. \first{We work under the assumptions of Theorem \ref{thm:stable}: 
$f(x)$ and $\frac{\partial f(x)}{\partial x_j} \ (j=1,\dots,n) $ are continuous near the origin,
$A(x)$ is continuous, and the pairs $ (A(x),B)$ and $(A(x),Q^{1/2})$  are, respectively, stabilizable and detectable
in a nonempty neighborhood of the origin.
}
The first approach, known as the \emph{offline–online} method, employs a first-order approximation of the solution and precomputes specific components to mitigate computational costs during real-time execution. The second approach employs an iterative scheme, referred to as the \emph{Newton–Kleinman} method, in which the Riccati solution from the previous time step is used as an initial guess for the subsequent iteration. These two techniques are evaluated in terms of computational efficiency and accuracy through a numerical experiment involving the optimal control of a nonlinear reaction-diffusion PDE.

\subsection{Offline–Online Approach}

Suppose that the semilinear form admits the decomposition
\begin{equation}
  A(x) = A_0 + \sum_{j=1}^{r} f_j(x) A_j = A_0 + \tilde{A}(x),
  \label{eq:A_decomp}
\end{equation}
where $A_0,A_j \in \mathbb{R}^{d \times d}$ and $f_j: \mathbb{R}^d \rightarrow \mathbb{R}$ for $j \in \{1,\dots, r\}$. This formulation is commonly encountered in the optimal control of PDEs, where the system dynamics include both a linear operator (e.g., diffusion or convection) and a nonlinear term, such as a reaction component.

Within this framework, the solution to the Riccati equation can be approximated as
\begin{equation}
  P(x) \approx P_0 + \sum_{j=1}^{r} P_j f_j(x) = P_0 + W(x),
  \label{eq:pi_approx}
\end{equation}
where:
\begin{itemize}
  \item \(P_0\) is computed by solving the CARE for the linearized system corresponding to \(A_0\):
    \begin{equation}
    A_0^\top P_0 + P_0 A_0 - P_0 S P_0 + Q = 0,
    \label{P0}
        \end{equation}
    with $
    S = B R^{-1}B^\top.
   $
  \item Each \(P_j\) is obtained by solving an associated Lyapunov equation
    \[
    P_j C_0 + C_0^\top P_j + Q_j = 0, \quad j=1,\ldots,r,
    \]
    with \(C_0 = A_0 - SP_0\) and \(Q_j = P_0 A_j + A_j^\top P_0\).
\end{itemize}

This methodology remains computationally feasible when the number of scalar nonlinear terms, $r$ is small. However, in many practical scenarios $r$ is comparable to the system dimension $d$, which can be arbitrarily large, making the direct implementation impractical due to the excessive computational burden.
When both $r$ and $d$ are large, an \emph{offline–online} strategy can be adopted to mitigate computational complexity. During the offline phase, the matrix $P_0$ is precomputed by solving \eqref{P0}. In the online phase, for a given state \(x\), the term $W(x)$ is computed by solving a single Lyapunov equation:
\begin{equation}
  W(x) C_0 + C_0^\top W(x) + Q(x) = 0,
  \label{eq:W_lyap}
\end{equation}

where 
$$
Q(x) = \sum_{j=1}^{r} Q_j f_j(x) = P_0 \tilde{A}(x) + \tilde{A}(x)^\top P_0.
$$
The resulting feedback law is then given by
\begin{equation}
  u_{O}(x) = -R^{-1}B^\top\Bigl( P_0 + W(x) \Bigr)x.
  \label{eq:feedback_offline_online}
\end{equation}

\begin{algorithm}
\caption{Offline-online SDRE}
\label{alg:off_onl}{
\begin{algorithmic}[1]
\Statex{\textbf{Input:} $A(\cdot), B, Q,R$, initial conditions $ y_0$, time discretization $\{t_i\}_{i=0}^{N_T}$ }
\Statex{\textbf{Output:} Controlled trajectory $\{y_O(t_i)\}_{i=0}^{N_T}$}
\Statex{}
\Statex \textbf{Offline phase}
\State Compute the solution $P_0$ from \eqref{P0}
\Statex{}
\Statex \textbf{Online phase}
\State{$y_O(t_0):= y_0$}
\For{$n = 0, \dots, N_T-1$}
\State{Compute $W(y_O(t_{n}))$ from \eqref{eq:W_lyap} }
\State{Build the control $u_O(y_O(t_{n}))$ from \eqref{eq:feedback_offline_online}}
\State{Integrate the system \eqref{semilinear} to obtain $y_O(t_{n+1})$}       
\EndFor 
\end{algorithmic}}
\end{algorithm}

The method is sketched in Algorithm \ref{alg:off_onl}.
This approach is attractive in real–time applications because it decouples the heavy computation (solved offline) from the online computation that is reduced to the resolution of one Lyapunov equation per time step. However, the stability of the closed-loop matrix $C(x) = A(x)-S\Bigl(P_0+W(x)\Bigr)$ is not inherently guaranteed. The following proposition provides a rigorous framework under which this approach ensures the stability of the closed-loop matrix.

\begin{proposition}\label{prop:stability}
Given the semilinear form \eqref{eq:A_decomp} and $P_0$ the solution of the CARE \eqref{P0}. Let $W(x)$ be the unique solution of the Lyapunov equation \eqref{eq:W_lyap} and let $C_0$ be diagonalizable with decomposition $C_0 = V \Lambda V^{-1}$.

Then if
\begin{equation}\label{eq:bound2}
\left\|\tilde{A}(x)\right\| \left(1+\|S\|\frac{M^2}{\alpha}\left\|P_0\right\|\right) < \alpha,
\end{equation}
where $\alpha = \min_{\lambda \in \sigma(C_0)} |\operatorname{Re}(\lambda)|$, with $\operatorname{Re}(\cdot)$ denoting the real part operator, and $M = cond(V) = \|V \| \| V^{-1} \|$,
the closed-loop matrix $C(x) = A(x)-S\Bigl(P_0+W(x)\Bigr)$ is Hurwitz \first{(i.e.\ with eigenvalues with negative real part)}.
\end{proposition}

\begin{proof}

The solution $W(x)$ of the Lyapunov equation \eqref{eq:W_lyap} can be expressed as
$$
W(x)=- \int_{0}^{\infty} e^{C_0^\top t} Q(x) e^{C_0 t}\,dt.
$$

Since $P_0$ is the unique solution of the CARE \eqref{P0}, the corresponding closed-loop matrix $C_0$ is Hurwitz. Consequently,
$$
\|e^{C_0 t}\| = \|V e^{\Lambda t } V^{-1}\| \le M e^{-\alpha t},
$$
with $\alpha = \min_{\lambda \in \sigma(C_0)} |\operatorname{Re}(\lambda)|$.
Applying this bound to the integral representation of $W(x)$, we obtain
$$
\|W(x)\| \le M^2 \| Q(x) \| \int_0^{+\infty}  e^{-2 \alpha t} \,dt = \frac{M^2}{2\alpha}  \| Q(x) \|. 
$$
Furthermore, we note that 
$$
\|Q(x) \| \le 2 \|P_0\| \|\tilde{A}(x)\|.
$$

Noting that $C(x) = C_0 + C_1(x) = C_0+ \tilde{A}(x)-S W(x)$,
a sufficient condition for the closed-loop matrix $C(x)$ to be stable is that the norm of $C_1(x)$ remains strictly less than the stability margin $\alpha$ of $C_0$. By applying the derived bound on $\|W(x)\|$ and $\|Q(x)\|$, this requirement leads to the computable stability criterion given by \eqref{eq:bound2}.
 
\end{proof}

The result presented above establishes a condition under which the closed-loop matrix $C(x)$ remains Hurwitz, thereby ensuring the synthesis of a stabilizing feedback law. However, if the norm of the nonlinear perturbation $\|\tilde{A}(x)\|$ is excessively large, the stability of the closed-loop matrix may be compromised. In such cases, the resulting controlled trajectory may fail to converge to the desired equilibrium. In the following section, we introduce an alternative approach based on an iterative scheme, which, although possibly requiring the solution of additional Lyapunov equations, ensures the stability of the closed-loop matrix under less restrictive conditions.

\subsection{Newton–Kleinman Method}
An alternative approach relies on the resolution of the full SDRE \eqref{sdre} at each state of the dynamical system, thereby avoiding the approximation error introduced by the decomposition in \eqref{eq:pi_approx}. Since the SDRE must be solved at multiple time steps, an efficient strategy involves employing an iterative scheme, such as the Newton–Kleinman (NK) method.  At a given time step $t_n$ for a fixed state $x$, the NK method refines an initial guess \(P^{(0)}(x)\) iteratively, updating it according to the recurrence relation
\begin{equation}
  P^{(k+1)}(x) = P^{(k)}(x) + \Delta P(x),
  \label{eq:NK_update}
\end{equation}
where the correction \(\Delta P(x)\) is computed by solving the Lyapunov equation
\begin{equation}
  (A(x)-SP^{(k)}(x) )^\top \Delta P(x) + \Delta P(x) (A(x)-SP^{(k)}(x) )  = -\mathcal{R}(P^{(k)},x),
  \label{eq:NK_lyap}
\end{equation}
with the residual
\[
\mathcal{R}(P^{(k)},x) = A^\top(x) P^{(k)}(x) + P^{(k)}(x) A(x) - P^{(k)}(x) S P^{(k)}(x) + Q.
\]
The procedure is repeated until the residual $\mathcal{R}(P^{(\tilde{k})},x)$ falls below a prescribed threshold or the maximum number of iterations is reached. Subsequently, we assign $P_{n}(x):=P^{(\tilde{k})}(x)$.
In the context of time–varying nonlinear systems, one can \emph{warm start} the NK iteration by using the Riccati solution from the previous time step as the initial guess \(P^{(0)}\). The technique is outlined in Algorithm \ref{alg:CNK}.
This strategy, which leverages the temporal continuity of the Riccati solution to enhance computational efficiency and convergence, is referred to as the \emph{cascade Newton–Kleinman} (\texttt{C-NK}) method. In \cite{saluzzi2025dynamical}, the authors demonstrate that if the time step is sufficiently small, bounded by a computable threshold, then the solution from the previous iteration $P(y_S(t_{n-1}))$ stabilizes the pair $(A(y_S(t_n)),B)$, $i.e.$, the closed-loop matrix $A(y_S(t_n))-S P(y_S(t_{n-1}))$ is stable, ensuring the convergence of the NK scheme (see \cite{kleinman1968} for a precise formulation).

The warm start strategy can be applied only from the second time step onward, as the initial Riccati solution at the first time step cannot be obtained from a previous iteration. Consequently, the first instance of the SDRE \eqref{sdre} is solved using a direct method, such as the MATLAB function \texttt{icare}. This step is analogous to the first phase of the offline-online method, where the initial Riccati equation \eqref{P0} must be solved. \first{The method does not require the input matrix $B$ to be constant; it can be applied as well when $B$ is state dependent, namely $B = B(x)$.
}

The key distinction between these two methods lies in their computational requirements: while the offline-online approach necessitates solving a single Lyapunov equation per time step, the cascade Newton–Kleinman scheme may require multiple iterations to achieve convergence. The availability of a suitable initial guess in the \texttt{C-NK} method not only accelerates convergence but also ensures a stabilizing initial condition, provided that the time step is sufficiently small.
In the following section, we present a numerical experiment demonstrating the effectiveness of this technique.

\begin{algorithm}
\caption{\texttt{C-NK} SDRE}
\label{alg:CNK}{
\begin{algorithmic}[1]
\Statex{\textbf{Input:} $A(\cdot), B, Q,R$, initial conditions $ y_0$, time discretization $\{t_i\}_{i=0}^{N_T}$}
\Statex{\textbf{Output:} Controlled trajectory $\{y_S(t_i)\}_{i=0}^{N_T}$}
\State{$y_S(t_0):= y_0$}
\State{Compute $P_0 = P(y_0)$ from \eqref{sdre}}
\State{Compute the control $u_S(y_0)$ from \eqref{control_sdre}}
\State{Integrate the system \eqref{semilinear} to obtain $y_S(t_{1})$}
\For{$n = 1, \dots, N_T-1$}
\State{Compute $P_n$ from \eqref{eq:NK_update}-\eqref{eq:NK_lyap} with guess $P_{n-1}$ }
\State{Build the control $u_S(y(t_{n}))$ from \eqref{control_sdre}}
\State{Integrate the system \eqref{semilinear} to obtain $y_S(t_{n+1})$}       
\EndFor 
\end{algorithmic}}
\end{algorithm}

\subsection{Numerical Example}

Throughout this discussion, we denote by $\chi_{\omega}(\cdot)$ the indicator function over the subdomain $\omega \subset \mathbb{R}$. The interval $\omega_c$ specifies the region where the control input is applied, while $\omega_o$ designates the observation domain where the system dynamics are monitored.
We consider the nonlinear Zeldovich-type equation with homogeneous Neumann boundary conditions:
\begin{equation*}
\left\{ \begin{array}{l}
\partial_t y(t,x) = \sigma \partial_{xx} y(t,x) + \nu y(t,x)+ \mu y(t,x)^2 (1-y(t,x)) + u(t,x)\chi_{\omega_c}(x),  \\
 y(0,x)=y_0(x),
\end{array} \right.
\label{AC}
\end{equation*}
with $x \in [0,1]$ and $t \in (0,+\infty)$. The corresponding cost functional is given by
$$
 J(u,y_0) = \int_0^{+\infty} \int_{\omega_o} | y(t,x)|^2 \, dx + \tilde{\gamma} \int_{\omega_c} | u(t,x) |^2 \, dx \, dt.
$$
By discretizing the PDE using a finite difference scheme with $d$ grid points, we obtain the following system of ODEs 
\begin{align*}
\dot{y}(s)=A(y) y(s) + u(s),
\end{align*}
with
$$
A(y) = A_0 +  \mu \, diag(y-y \odot y),\quad y \in \mathbb{R}^d,
$$
$$
A_0 = \sigma A^{\Delta} + \nu I_d.
$$
Here, $\odot$ denotes the Hadamard product, $I_d \in \mathbb{R}^{d \times d}$ is the identity matrix and $A^{\Delta}$ is the discrete Laplacian matrix arising from the finite difference approximation with Neumann boundary conditions. We set the discretization dimension to $d =100$ and initialize the system with $y_0(x) = \cos (\pi x)$. The dynamical system is discretized using a semi-implicit first-order integration scheme, where the diffusion term is treated implicitly, while the reaction and control terms are handled explicitly. A fixed time step of $\Delta t = 0.02$ is employed, with the final simulation time set to $T=4$, leading to a total of $200$ time steps. We conduct a comparative analysis of the offline-online method (Algorithm \ref{alg:off_onl}), the \texttt{C-NK} method (Algorithm \ref{alg:CNK}), and the "\texttt{icare} approach", which is based on Algorithm \ref{alg:sdre} and utilizes the MATLAB function \texttt{icare} to solve the Riccati equations. In the \texttt{C-NK} method, the stopping criterion for the residual is set to $10^{-5}$. Two different parameter configurations are considered:
\begin{enumerate}
    \item $\sigma = 0.2$, $\tilde{\gamma} = 0.01$, $\nu = 0.5$ where the control is applied in $\omega_c= [0.2,0.5] $ and the observation is restricted to $ \omega_o = [0.5,0.7]$, leading to a partial-domain control and monitoring setup.
    \item $\sigma = 10^{-2}$, $\tilde{\gamma} = 0.1$, $\nu = 0.5$ with full-domain control and observation regions $\omega_c = \omega_o = [0,1]$, meaning that both the control and observation mechanisms cover the entire spatial domain.
\end{enumerate}
These two scenarios exhibit distinct behaviors in the decay of the singular values of the Riccati solution. Fixing $\mu=1$, we solve the corresponding SDRE at the initial condition $y_0$. The decays of the singular values for both test cases are illustrated in the left panel of Figure \ref{pics:sing_value}. In the first case, the singular values exhibit a rapid decay, suggesting a low-rank structure in the solution. In the second case, the singular values decay slowly, indicating that the Riccati solution is full-rank. For high-dimensional problems, these two cases require different numerical treatments. For low-rank solutions, there exists a substantial body of literature on efficient methods for solving low-rank Lyapunov and Riccati equations. For further details, we refer the reader to \cite{simoncini2014two,simoncini2016computational,benner2008numerical}. Conversely, in the full-rank setting, the sparsity and potential banded structure of the solution can be effectively leveraged, as illustrated in the right panel of Figure \ref{pics:sing_value}. For further details, the reader is referred to \cite{palitta2018,massei2024data,haber2018sparsity} and the references therein. A comprehensive analysis and implementation of these advanced techniques fall beyond the scope of this study. Given the problem dimension considered, the Lyapunov equations are solved using the MATLAB function \texttt{lyap}, while the initial Riccati equation is computed employing the MATLAB command \texttt{icare}.

\begin{figure}[H]	
\begin{center}    
	\includegraphics[width=0.49\textwidth]{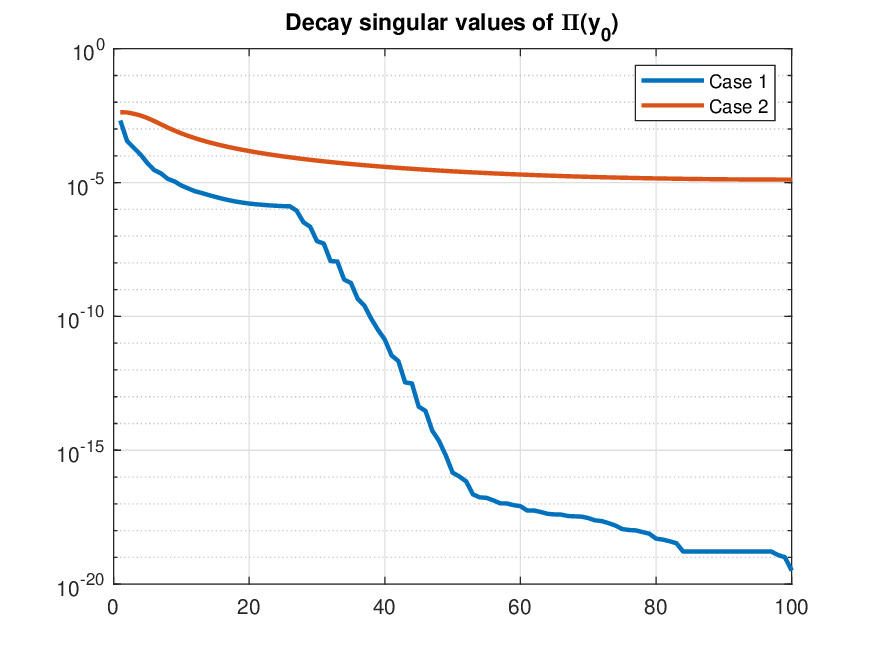}	
	\includegraphics[width=0.49\textwidth]{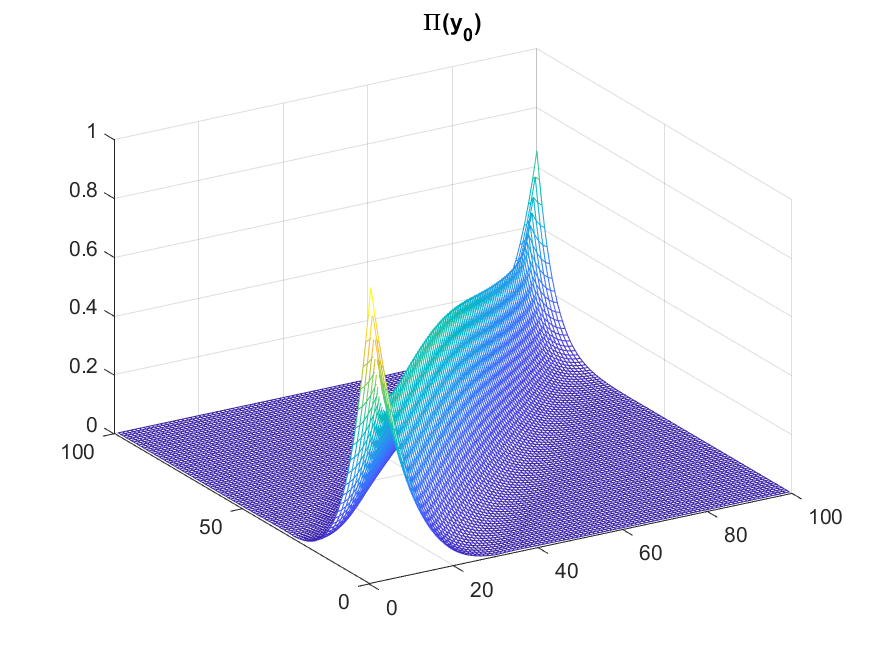}
\caption{Zeldovich example. \emph{Left}: Decay of the singular values of the Riccati solution $P(y_0)$ for the two different cases. \emph{Right}: Off diagonal decay for the solution $P(y_0)$ in the second case.}
\label{pics:sing_value}
\end{center}

\end{figure}

\vspace{0.5cm}

{\bf Case 1}

\vspace{0.2cm}

We first consider the scenario where the control and observation domains are given by $\omega_c= [0.2,0.5] $ and $ \omega_o = [0.5,0.7]$, respectively, with parameters $\sigma = 0.2$, $\tilde{\gamma} = 0.01$ and $\nu = 0.5$. The reaction parameter $\mu$ varies within the set $\{1,2\}$. Table \ref{table_AC_low_rank} presents the CPU time and the total cost obtained using different numerical approaches for the considered values of $\mu$. The results indicate that the most efficient method, achieving the lowest total cost, is the \texttt{C-NK} approach. While the SDRE method based on \texttt{icare} yields a total cost comparable to that of \texttt{C-NK}, being approximately 60 times slower. Conversely, the offline-online approach requires up to four times the CPU time and results in a higher total cost.  The superior computational efficiency of \texttt{C-NK} can be attributed to the fact that, for certain time steps, the previously computed solution exhibits a low residual. As a result, the resolution of a Lyapunov equation is not required at every step, reducing computational overhead.
In the case $\mu=2$, the control law generated by the offline-online method fails to stabilize the system, as illustrated in the right panel of Figure \ref{pics:AC_cost_low_rank}.
Furthermore, Figure \ref{pics:AC_cost_low_rank} demonstrates that the \texttt{C-NK} and \texttt{icare}-based methods exhibit similar performance trends.

\begin{table}[hbht]
\centering
\begin{tabular}{c|cc|cc|cc}    

        & \multicolumn{2}{c|}{offline-online}   & \multicolumn{2}{c|}{\texttt{C-NK}} & \multicolumn{2}{c}{\texttt{icare}} \\
$\mu$   & CPU time & Total cost & CPU time & Total cost & CPU time & Total cost 
\\\hline
1    & 0.98  &   3.10e-01 & 0.27   &  3.08e-01  &  15.7 & 3.08e-01\\

2 & 1.02  &   1.63e+03  & 0.25   &  8.56e-01   &  15.5 & 8.56e-01  \\
 \end{tabular}
  \caption{Zeldovich example (Case 1). Comparison of the performance in the computation of the controlled trajectory via the offline-online approach, \texttt{C-NK} method and via the \texttt{icare}-based SDRE.}
 \label{table_AC_low_rank}
\end{table}

\begin{figure}[H]	
\begin{center}    
	\includegraphics[width=0.49\textwidth]{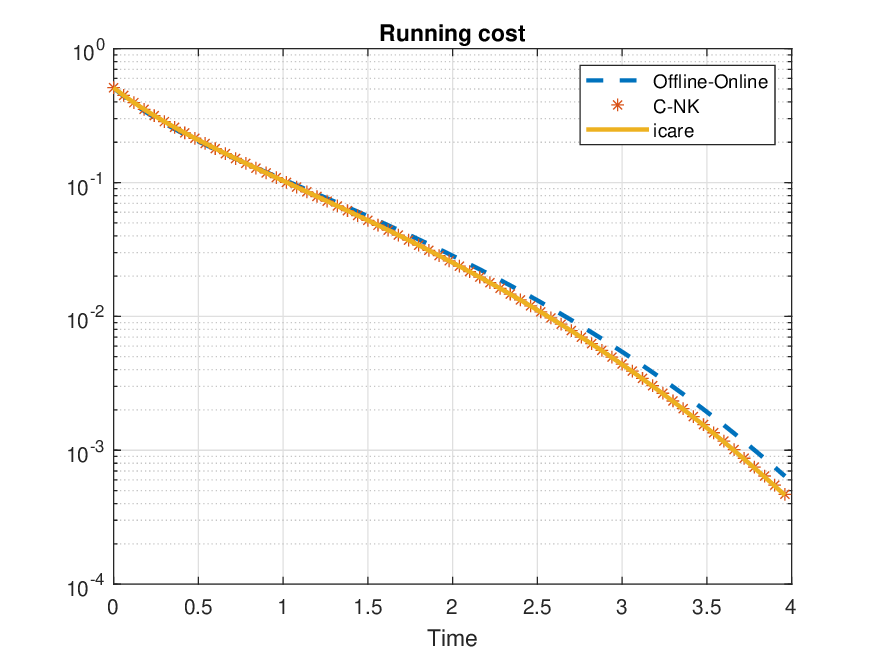}	
	\includegraphics[width=0.49\textwidth]{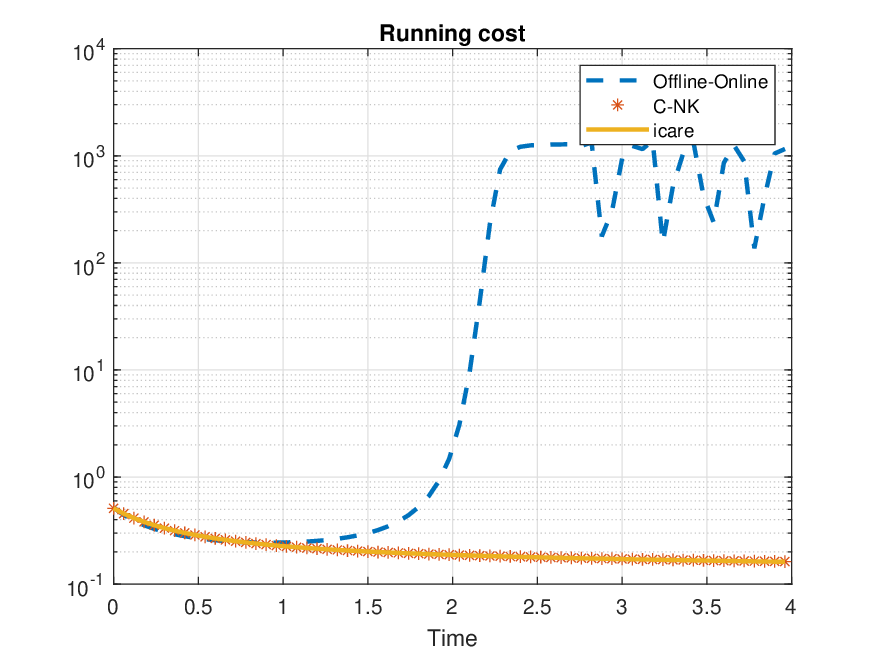}	
\caption{Zeldovich example (Case 1). Running cost for the controlled trajectories with the different techniques with $\mu =1$ (left) and $\mu=2$ (right).}
\label{pics:AC_cost_low_rank}
\end{center}

\end{figure}

\vspace{0.5cm}

{\bf Case 2}

\vspace{0.2cm}

In this second case, we consider the fully controlled and observed setting, where $\omega_c = \omega_o = [0,1]$. The parameters are set to $\sigma = 10^{-2}$, $\tilde{\gamma} = 0.1$ and $\nu = 0.5$, while the reaction coefficient $\mu$ varies within the set $\{1,2\}$. The results obtained for the three numerical techniques across different values of $\mu$ are summarized in Table \ref{table_AC_full_rank} and illustrated in Figure \ref{pics:AC_cost_full_rank}. For the choice $\mu = 2$, the final simulation time is reduced to 0.6, due to the fast divergence of the offline-online method.
The observed trends are consistent with those in the previous case. Once again, the \texttt{C-NK} method demonstrates the best performance in terms of both computational efficiency and total cost. The \texttt{icare}-based SDRE approach achieves the same total cost as \texttt{C-NK}, but its computational time is more than 40 times higher for the first choice of $\mu$, further emphasizing the efficiency of \texttt{C-NK}.
Regarding the offline-online method, for $\mu =1$ the total cost is comparable to that of the other two techniques. As in the previous case, increasing the reaction coefficient to $\mu = 2$ results in the failure of the offline-online approach to stabilize the system.

\begin{table}[hbht]
\centering
\begin{tabular}{c|cc|cc|cc}    

        & \multicolumn{2}{c|}{offline-online}   & \multicolumn{2}{c|}{\texttt{C-NK}} & \multicolumn{2}{c}{\texttt{icare}} \\
$\mu$   & CPU time & Total cost & CPU time & Total cost & CPU time & Total cost 
\\\hline
1    & 0.53  &   2.49e-01 & 0.46   &  2.38e-01  &  21.2 & 2.38e-01\\

2 & 0.16  &   3.57e+04 & 0.30   &  2.78e-01   &  8.11 & 2.78e-01  \\
 \end{tabular}
  \caption{Zeldovich example (Case 2). Comparison of the performance in the computation of the controlled trajectory via the offline-online approach, \texttt{C-NK} method and via the \texttt{icare}-based SDRE.}
 \label{table_AC_full_rank}
\end{table}

\begin{figure}[H]	
\begin{center}    
	\includegraphics[width=0.49\textwidth]{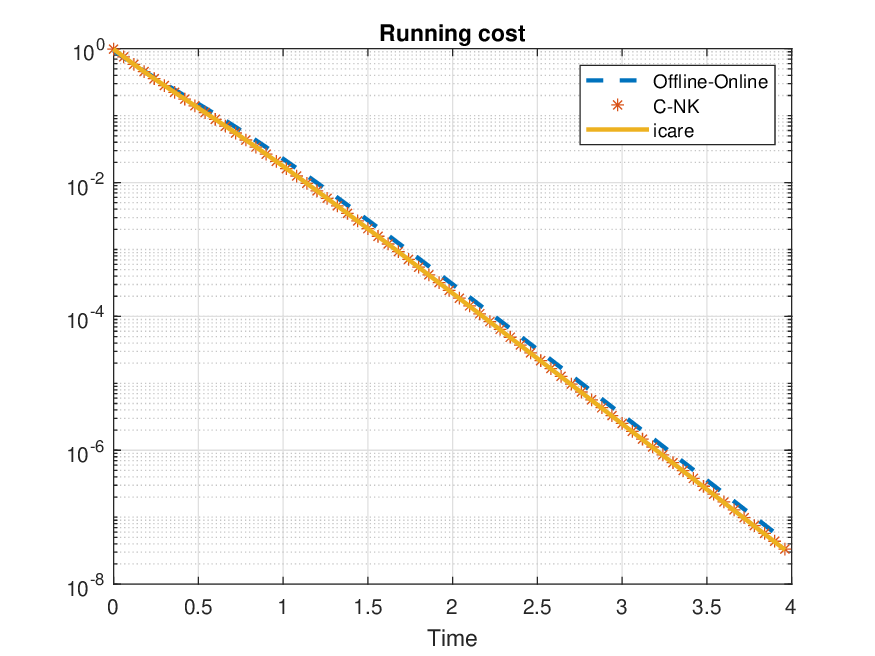}	
    \includegraphics[width=0.49\textwidth]{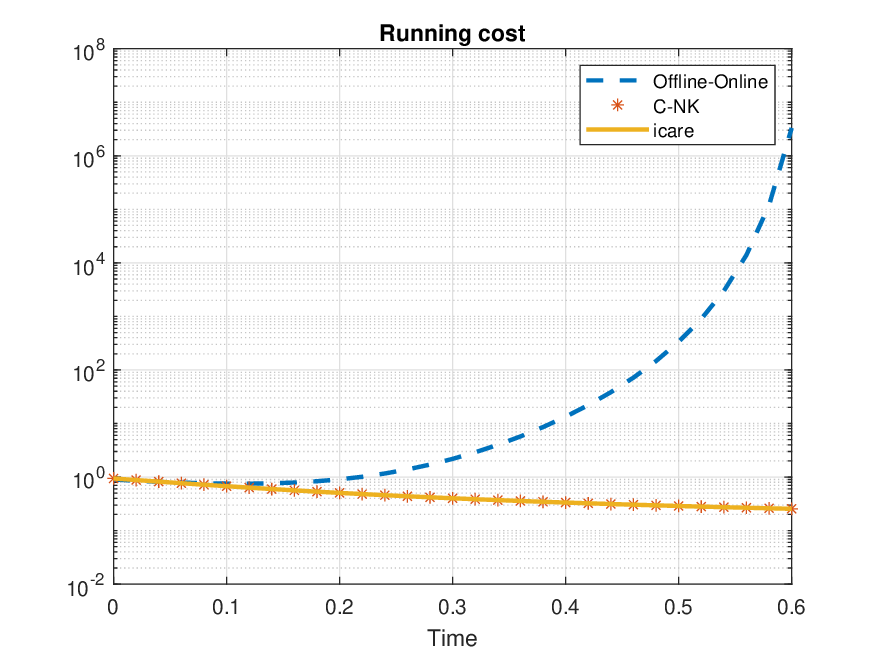}	
\caption{Zeldovich example (Case 2). Running cost for the controlled trajectories with the different techniques with $\mu =1$ (left) and $\mu=2$ (right).}
\label{pics:AC_cost_full_rank}
\end{center}
\end{figure}


\section{Conclusions}

This work provides an analysis of the State-Dependent Riccati Equation method for nonlinear optimal control, focusing on its theoretical foundations, numerical approximations, and computational efficiency. We examined the relationship between SDRE and the HJB equation, deriving error estimates based on the residual introduced by the SDRE approximation. Our analysis highlighted the role of the semilinear decomposition in influencing the accuracy of the SDRE approach and proposed an optimal decomposition strategy to minimize approximation errors.

From a computational perspective, we compared two numerical methods for solving the SDRE: the offline–online approach and the Newton–Kleinman iterative method. Numerical experiments demonstrated that while the offline–online approach offers a reduced computational burden, it may fail to stabilize the system under certain conditions. In contrast, the C-NK method consistently provided more accurate and stable controlled solutions, making it a preferable choice in real-time control applications.

Future work could explore more efficient numerical methods for solving the sequence of Riccati equations in high-dimensional settings, leveraging low-rank approximations, sparsity-preserving methods, or data-driven techniques. Additionally, extending the SDRE framework to stochastic control scenarios would be a valuable direction for further research.

\vspace{1cm}

{\bf Acknowledgements.} The work has been partially supported by Italian Ministry of
Instruction, University and Research (MIUR) (PRIN Project 2022 2022238YY5, “Optimal control problems: analysis, approximation”) and INdAM-research group GNCS (CUP E53C24001950001,
“Metodi avanzati per problemi di Mean Field Games ed applicazioni”).

\vspace{5mm}

{\bf Data Availability.} Data will be made available on request.

\vspace{1cm}

\bibliographystyle{spmpsci}
\bibliography{sample,biblio}
\end{document}